\documentclass{amsart}
\usepackage[utf8]{inputenc}
\usepackage{amssymb}
\usepackage{amsthm}
\usepackage{amsmath}
\usepackage{enumerate}
\usepackage{mathtools}
\usepackage{comment}
\usepackage{todonotes}
\usepackage{xcolor}\mathtoolsset{showonlyrefs}
\usepackage{hyperref}

\newtheorem{thm}{Theorem}[section]
 
 \newtheorem{lem}[thm]{Lemma}
 \newtheorem{prop}[thm]{Proposition}
 
 \theoremstyle{definition}
 \newtheorem{defn}[thm]{Definition}
  \newtheorem{ex}[thm]{Example}
 \theoremstyle{remark}
 
\title[Metric space valued Sobolev maps via weak* derivatives]{An approach to metric space valued Sobolev maps via weak* derivatives}
\author{Paul Creutz}
	\address{Department of Mathematics and Computer Science, University of Cologne, Weyertal 86-90, 50931 K\"oln, Germany.}
	\email{pcreutz@math.uni-koeln.de}

\author{Nikita Evseev}
	\address{Sobolev Institute of Mathematics, 4 Academic Koptyug avenue,  630090 Novosibirsk, Russia.
	}
	\email{evseev@math.nsc.ru}

\thanks{The first named author was partially supported by the DFG-grant SFB/TRR 191 ``Symplectic structures in Geometry, Algebra
and Dynamics."} 

\begin{document}

\maketitle
\begin{abstract}
We give a characterization of metric space valued Sobolev maps in terms of weak* derivatives. This corrects a previous result by Haj{\l}asz and Tyson.
\end{abstract}

\section{Introduction}
\subsection{Objective}
This article concerns possible definitions of the first-order Sobolev space $W^{1,p}(\Omega ;X)$ for an open subset $\Omega \subset \mathbb{R}^n$, a metric space $X$ and a coefficient $p\in (1,\infty)$. 
Since the early 1990's several definitions of such Sobolev spaces have been proposed in \cite{KS:93,  Haj:96, Reshetnyak97,Che:99, Sha:00, HK:00, AGS:14}. 
Many of these make sense when $\Omega$ is an arbitrary metric measure space and, in such generality, the arising Sobolev space may depend on the chosen definition. 
However, for bounded domains $\Omega\subset \mathbb{R}^n$, all of these definitions are equivalent, see  
\cite{Reshetnyak2004, AGS:13, HKST2015}. The mentioned characterizations of $W^{1,p}(\Omega ;X)$ take very different approaches that mostly involve slightly advanced concepts such as energy, modulus of curve families or Poincaré inequalities. Hence, from the point of view of classical analysis, all these characterizations might either seem a bit complicated or at least not very straightforward. 
Another definition of the Sobolev space $W^{1,p}(\Omega ;X)$ was proposed in~\cite{HT:08} which is more similar to the traditional definition of classical Sobolev spaces in terms of weak derivatives. 
Our first main result, Theorem~\ref{thm:triv} below, however shows that for technical reasons the space $W^{1,p}(\Omega ;X)$ as introduced in \cite{HT:08} is essentially empty. The main objective of this article is then to propose a variation on the definition from \cite{HT:08} and show that this new definition indeed gives an equivalent characterization of the Sobolev spaces introduced in \cite{KS:93,  Haj:96, Reshetnyak97,Che:99, Sha:00, HK:00, AGS:14}.
\subsection{Definitions and main results}
If $X$ is a Riemannian manifold then, by Nash's theorem, there is a Riemannian isometric embedding $\iota \colon X \to \mathbb{R}^N$. In this case $W^{1,p}(\Omega;X)$ can be defined as the set of those functions $f\colon \Omega \to X$ for which the composition~$\iota \circ f$ lies in the classical Sobolev space $W^{1,p}(\Omega; \mathbb{R}^N)$. Similarly one can embed any metric space $X$ isometrically into some Banach space $V$ as to force a linear structure on the target space. 
For example every separable metric space embeds isometrically into $\ell^\infty$ by means of the Kuratowski embedding. Thus it is natural to first define Sobolev functions with values in the Banach space $V$ and then $W^{1,p}(\Omega ; X)$ as the subspace of those functions in $W^{1,p}(\Omega ; V)$ that take values in~$X$ with respect to the fixed embedding. The following definition of Banach space valued Sobolev functions goes back to~\cite{Sob:63}.

\begin{defn}
\label{fd}
Let $V$ be a Banach space and $p\in [1,\infty)$. The space $ L^{p}(\Omega; V)$ consists of those functions $f\colon \Omega\to V$ that are measurable and essentially separably valued, and for which the function $x \mapsto ||f(x)||$ lies in $L^p(\Omega)$.\\
A function $f$ lies in the Sobolev space $W^{1,p}(\Omega ; V)$ if $f\in L^p(\Omega ;V)$ and for every $j=1, \dots , n$ there is a function $f_j\in L^p(\Omega;V)$ such that
\begin{equation}
\int_\Omega\frac{\partial \varphi}{\partial x_j}(x)\cdot f(x)\ \textrm{d}x 
= -\int_\Omega\varphi(x)\cdot f_j(x) \ \textrm{d}x \quad \text{ for every } \varphi \in C^\infty_0(\Omega)
\end{equation}
in the sense of Bochner integrals.
\end{defn}
 It was claimed in \cite{HT:08} that if $Y$ is separable then $W^{1,p}(\Omega ; Y^*)$ is equal to the Reshetnyak-Sobolev space $R^{1,p}(\Omega ; Y^*)$ introduced in \cite{Reshetnyak97}. This would imply that the Sobolev space $W^{1,p}(\Omega; X)$, defined in terms of Definition~\ref{fd} and the Kuratowski embedding $\kappa \colon X \to \ell^\infty$, is the same as the Sobolev spaces introduced in \cite{Reshetnyak97, KS:93, Sha:00, Haj:96, HK:00, Che:99, AGS:14}. 
Unfortunately, it has recently been observed in \cite{CJPA2020} that there is a subtle measurability-related mistake in the proof of the equality and indeed $W^{1,p}(\Omega ;Y^*)$ equals $R^{1,p}(\Omega ; Y^*)$ only if $Y^*$ has the Radon-Nikod\'ym property. For the sake of defining metric space valued Sobolev maps this is potentially problematic because many spaces of geometric interest, such as the Heisenberg group or even~$\mathbb{S}^1$ (equipped with the angular metric), do not isometrically embed into a Banach space which has the Radon-Nikod\'ym property, see \cite{CK:06} and \cite[Remark~4.2]{Cre:20}. 
Our first main result shows that indeed $W^{1,p}(\Omega ;X)$, as defined in \cite{HT:08} in terms of Definition~\ref{fd} and the Kuratowski embedding, is always trivial, and hence $W^{1,p}(\Omega;X)$ is not equal to $R^{1,p}(\Omega ;X)$ for any geometrically interesting space $X$.

\begin{thm}
\label{thm:triv}
Let $\Omega \subset \mathbb{R}^n$ be a bounded domain, $X$ be a complete separable metric space and $p\in [1,\infty)$. Then a function $f\colon \Omega \to X$ lies in $ W^{1,p}(\Omega;X)$ if and only if it is almost everywhere constant.
\end{thm}
There is a number of articles subsequent to \cite{HT:08} that have worked with this definition of metric space valued Sobolev maps, see \cite{Haj:09, WZ:09, Haj:11, BMT:13, BHW:14, HST:14, DHLT:14}. In particular, important results such as \cite[Theorem~1.2]{WZ:09}, \cite[Theorem~1.4]{Haj:11} or \cite[Theorem~1.9]{HST:14} are formally not correct as stated. 
To fix this technical problem, instead of Definition~\ref{fd}, we suggest the following one.
\begin{defn}
\label{fd2}
Let $V^*$ be a dual Banach space and $p\in [1,\infty)$. The space $ L^{p}_*(\Omega; V^*)$ consists of those functions $f\colon \Omega\to V^*$ that are weak* measurable and for which the function $x \mapsto ||f(x)||$ lies in $L^p(\Omega)$.\\
A function $f$ lies in the Sobolev space $W^{1,p}_*(\Omega ; V^*)$ if $f\in L^p(\Omega ;V^*)$ and for every $j=1, \dots , n$ there is a function $f_j\in L^p_*(\Omega;V^*)$ such that
\begin{equation}
\label{eq:def-w*}
\int_\Omega\frac{\partial \varphi}{\partial x_j}(x)\cdot f(x)\ \textrm{d}x 
= -\int_\Omega\varphi(x)\cdot f_j(x) \ \textrm{d}x \quad \text{ for every } \varphi \in C^\infty_0(\Omega)
\end{equation}
in the sense of Gelfand integrals.
\end{defn}

The main difference between $W^{1,p}_*$ and $W^{1,p}$ is that for $W^{1,p}_*$ the weak derivatives do not need to be measurable and instead one only assumes weak* measurability. In particular, the functions $f_j$ in Definition~\ref{fd2} do not need to be Bochner integrable. Our second main result shows that $W^{1,p}_*$ indeed gives the right Sobolev space.

\begin{thm}\label{theorem:W=R}
Let $\Omega\subset \mathbb{R}^n$ be open, $Y$ be a separable Banach space, and $p\in [1,\infty)$. Then 
\[
W^{1,p}_*(\Omega;Y^*)=R^{1,p}(\Omega;Y^*).
\]
\end{thm}
Thus, for a bounded $\Omega$ and a separable metric space $X$, one can define $W^{1,p}_*(\Omega; X)$ as the set of those functions $f \colon \Omega \to X$ such that $\kappa \circ f\in W^{1,p}_*(\Omega ; \ell^\infty)$ and deduce that
\[
W^{1,p}_*(\Omega;X)=R^{1,p}(\Omega;X).
\]
We believe that essentially all results in the articles \cite{HT:08, Haj:09, WZ:09, Haj:11, BMT:13, BHW:14, HST:14, DHLT:14} become true if one respectively replaces $W^{1,p}(\Omega; X)$ by $W^{1,p}_*(\Omega; X)$ and that the proofs apply up to straightforward adjustments.

An advantage of our definition of $ W^{1,p}_*(\Omega; X)$ over the other equivalent definitions of metric space valued Sobolev maps is that it gives a characterization in terms of actual linear differentials and not just upper gradients, metric differential seminorms or alike. It might seem that such linear differentials are somewhat artificial in the context of general metric target spaces. However, indeed there are some nice arguments and constructions that heavily rely on this sort of objects, see e.g.\ \cite{Haj:11, DHLT:14, AK:00, Iva:08}.
\subsection{Organization}
First in Section~\ref{sec:2} we will go through some auxiliary results and definitions concerning the calculus of functions with values in Banach spaces. More precisely, in Sections~\ref{sec:2.1} and \ref{sec:2.2} we discuss different notions concerning measurability and integrals of Banach space valued functions. Then in Section~\ref{sec:abs-curves} we study some basic properties of the weak* derivatives of absolutely continuous curves in dual-to-separable Banach spaces. Section~\ref{sec:3} is dedicated to Sobolev maps with values in Banach spaces and more particularly the proof of Theorem~\ref{theorem:W=R}. 
To this end we will consider an auxiliary space $R^{1,p}_*(\Omega ; Y^*)$ whose definition interpolates between the definitions of  $R^{1,p}(\Omega ; Y^*)$ and $W^{1,p}_*(\Omega ; Y^*)$.  
In Sections~\ref{sec:3.1} and~\ref{sec:3.2} we then respectively prove the equalities $R^{1,p}_*=R^{1,p}$ and $R^{1,p}_*=W^{1,p}_*$. The more original part here is the proof of the equality $R^{1,p}_*=R^{1,p}$ since  
the proof of $R^{1,p}_*=W^{1,p}_*$ is very much along the lines of the intended proof of $W^{1,p}=R^{1,p}$ in \cite{HT:08}. In the final Section~\ref{sec:4} we discuss Sobolev functions with values in a metric space~$X$. First in Section~\ref{sec:4.1} we shortly introduce the Sobolev spaces $W^{1,p}_*(\Omega ;X)$. Then in Section~\ref{sec:4.2} we focus on $W^{1,p}(\Omega ;X)$ and prove Theorem~\ref{thm:triv}. The proof here is a slightly involved argument that exploits the strange analytic properties of the Kuratowski embedding.

\subsection{Acknowledgements and remarks} We want to thank Alexander Lytchak and Elefterios Soultanis for helpful comments. Also we are grateful to Piotr Haj{\l}asz who has made the contact among the two of us and with the authors of \cite{CJPA2020}. In this context we also learned that the authors of \cite{CJPA2020} are currently working on related questions concerning Sobolev functions with values in Banach spaces.
\section{Calculus of Banach space valued functions}
\label{sec:2}
During this section let $E \subset \mathbb{R}^n$ be Lebesgue measurable and $V$ be a Banach space.
\subsection{Measurability of Banach space valued functions}
\label{sec:2.1}
We call a function $f\colon E \to V$ \emph{measurable} if it is measurable with respect to the Borel $\sigma$-algebra on~$V$ and the $\sigma$-algebra of Lebesgue measurable subsets on $E$. It is called \emph{weakly measurable} if $x\mapsto \langle v^*, f(x)\rangle$ defines a measurable function $E\to \mathbb{R}$ for every $v^* \in V^*$ and \emph{essentially separably valued} if there is a null set $N\subset E$ such that $f(E \setminus N)$ is separable. Trivially measurability implies weak measurability. If additionally one assumes that $f$ is essentially separably valued then, by Pettis' measurability theorem, also the converse implication holds, see e.g.\ \cite[Section 3.1]{HKST2015}. In general however, weakly measurable functions do not need to be measurable, see \cite[Remark 3.1.3]{HKST2015}.

A function $f\colon E \to V$ is called \emph{approximately continuous} at $x \in E$ if for every $\varepsilon >0$ one has
\[
\lim_{r\downarrow 0}\frac{\mathcal{L}^n\left(\left\{y \in B(x,r)\cap E \ \colon \  ||f(y)-f(x)||\geq \varepsilon \right\}\right)}{\mathcal{L}^n\left(B(x,r)\right)}=0.
\]
The following characterization of measurability will be important in the proof of Theorem~\ref{thm:triv}.
\begin{thm}[\cite{Fed:65}, Theorem 2.9.13]
\label{thm:approx}
Let $f\colon E \to V$ be essentially separably valued. Then $f$ is measurable if and only if $f$ is approximately continuous at a.e.\ $x\in E$.
\end{thm}

A function $f\colon E \to V^*$ is called \emph{weak* measurable} if $x \mapsto \langle v , f (x)\rangle$ defines a measurable function $E \to \mathbb{R}$ for every $v \in V$. We will need the following slight strengthening of Pettis' theorem.
\begin{lem}
\label{lem:pettis}
Let $f\colon E \to V^*$ be essentially separably valued. Then $f$ is measurable if and only if $f$ is weak* measurable.
\end{lem}
\begin{proof}
Clearly measurable functions are weak* measurable. So we only prove the other implication. By assumption there is a null set $N\subset E$ such that $f(E \setminus N)$ is separable. Let $D=\{v_1^*,v_2^*,\dots\}$ be a countable dense subset in $f(E\setminus N)$. Then $D-D$ is a countable dense subset of the difference set $f(E \setminus N)-f(E \setminus N)$.  By definition of the dual norm for every $i,j\in \mathbb{N}$ there is a sequence $(v^{ij}_k)_{k\in \mathbb{N}}$ of unit vectors in $V$ such that \[\langle v^{ij}_k, v_i^*-v_j^*\rangle \to ||v_i^*-v_j^*|| \quad \textnormal{as } k\to \infty.\]
Thus, it follows from the weak* measurability of $f$ that for every $i\in \mathbb{N}$ the function
\[
x \mapsto ||f(x)-v_i^*|| =\sup_{j,k \in \mathbb{N}} \langle v^{ij}_k, v_i^*-f(x)\rangle
\]
is measurable. In particular, $f^{-1}(B)$ is measurable for every open ball $B\subset V^*$ with center in $D$. 

Let $U\subset V^*$ be open. Then there is a countable collection $(B_i)_{i\in \mathbb{N}}$ of balls in $V^*$ with centers in $D$ such that 
\begin{equation}
f(E\setminus N)\cap U =f(E\setminus N) \cap \left( \bigcup_{i\in \mathbb{N}} B_i \right)
\end{equation}
and hence
\begin{equation}
\label{eq:U-meas}
f^{-1}(U)\cup N=\left(\bigcup_{i\in \mathbb{N}} f^{-1}(B_i)\right)\cup N.
\end{equation}
Since $\bigcup_{i\in \mathbb{N}} f^{-1}(B_i)$ is Lebesgue measurable and $N$ is a null set, \eqref{eq:U-meas} implies that $f^{-1}(U)$ is Lebesgue measurable. 
The open subsets generate the Borel $\sigma$-algebra of~$V$, so we conclude that $f$ is measurable.
\end{proof}
\subsection{Integrals of Banach space valued functions}
\label{sec:2.2}
A function $f\colon E \to V$ is called \emph{simple} if there are measurable subsets $E_1, \dots , E_k$ of $E$ and vectors $v_1, \dots , v_k$ in $V$ such that $f=\sum^k_{i=1} \chi_{E_i}\cdot v_i$. If $f$ is simple and all the subsets $E_i$ are of finite $\mathcal{L}^n$-measure, then $f$ is called \emph{integrable} and one defines the \emph{integral} of $f$ as
\[
\int_E f(x)\ \textrm{d} x:= \sum^k_{i=1} \mathcal{L}^n(E_i)\cdot v_i.
\]
A function $f\colon E \to V$ is called \emph{Bochner integrable} if there are integrable simple functions $(f_k\colon E \to V)_{k\in \mathbb{N}}$ such that 
\[
\lim_{k \to \infty}\int_E ||f_k(x)-f(x)|| \ \textrm{d} x = 0
\]
The  \emph{Bochner integral} of such Bochner integrable function $f$ is defined as
\[
\int_{E} f(x) \ \textrm{d} x:=\lim_{k\to \infty} \int_{E} f_k(x) \ \textrm{d}x.
\]
Indeed, a function $f$ is Bochner integrable if and only it lies in the space $L^1(E ;V)$ introduced in Definition~\ref{fd}, see \cite[Proposition 3.2.7]{HKST2015}. Furthermore, if $f$ is Bochner integrable and $v^*\in V^*$ then $x \mapsto \langle v^*, f(x) \rangle$ is integrable and 
\begin{equation}
\label{eq:lin}
\left\langle v^*, \int_E f(x) \ \textrm{d} x\right\rangle =\int_E \langle v^*, f(x) \rangle \ \textrm{d} x.
\end{equation}

The Bochner integral is arguably the most popular notion concerning integrals of Banach space valued functions. However, its limitation to essentially separably valued measurable functions is somewhat to rigid for our purposes. Instead we will often work with the so-called Gelfand integral which is a weak* variant of the more well-known Pettis integral that is defined for weakly measurable functions. It goes back to \cite{Gel:36} and can be defined in terms of the following lemma. See also \cite[p.~53]{DU:1977}.
\begin{lem}\label{lemma:Gelfand-integral}
Let $f\colon E\to V^*$ be a weak* measurable function such that for every $v\in V$ the function $x \mapsto \langle v,f(x)\rangle$ lies in $L^1(E)$.
Then there is a unique vector $v^*_f\in V^*$ such that
\[
\langle v, v^*_f \rangle = \int_E\langle v, f(x) \rangle\ \textrm{d}x \quad \text{ for every } v\in V.
\]
\end{lem}
\begin{proof}
First we claim that the operator $T\colon V\to L^1(E)$ defined by $Tv = \langle v, f\rangle$ is continuous. 
To this end let $(v_k, Tv_k)_{k\in \mathbb{N}}$ belong to the graph of $T$. 
Suppose that $v_k\to v$ in $V$ and $Tv_k\to g$ in $L^1(E)$.
Then there is a subsequence $(Tv_{k_m})_{m\in \mathbb{N}}$ which converges a.e.\ on $E$ to $g$. In particular
\[
g(x) = \lim_{m\to\infty}Tv_{k_m}(x) = \lim_{m\to\infty}\langle v_{k_m}, f(x)\rangle
= \langle v, f(x)\rangle = (Tv)(x)
\]
for a.e.\  $x\in\Omega$. Hence the linear operator $T$ has a closed graph and the closed graph theorem implies that $T$ is continuous. 

Thus for every $v\in V$ one has
\[
\left| \int_E\langle v, f(x) \rangle\ \textrm{d}x \right| \leq \|Tv\| \leq \|T\|\cdot\|v\|.
\]
This shows that the functional $v_f^*$ given by $v_f^*(v):= \int_E\langle v, f(x) \rangle\ \textrm{d}x$ is continuous and hence completes the proof.
\end{proof}
Functions $f\colon E \to V^*$ that meet the assumptions of Lemma~\ref{lemma:Gelfand-integral} are called \emph{Gelfand integrable} and for such $f$ the arising functional $v^*_f$  is called the \emph{Gelfand integral} of $f$. By \eqref{eq:lin} and Lemma~\ref{lemma:Gelfand-integral}, if $f\colon E \to V^*$ is Bochner integrable then $f$ is Gelfand integrable and $\int_E f(x) \ \textrm{d} x=v^*_f$. Hence we will not create ambiguity when we also denote Gelfand integrals by $\int_E f(x) \ \textrm{d} x$ instead of $v^*_f$. Note that if $\Omega \subset \mathbb{R}^n$ is open and $f\in L^{p}_*(\Omega; V^*)$, then $\varphi\cdot f$ is Gelfand integrable for every $\varphi \in C^\infty_0(\Omega)$ and hence the Gelfand integrals that appear in Definition~\ref{fd2} are well-defined.
\subsection{Absolutely continuous curves in Banach spaces}
\label{sec:abs-curves}
Recall that a function $f \colon [a,b] \to \mathbb{R}$ is called \emph{absolutely continuous} when it satisfies the fundamental theorem of calculus. That is when $f$ is differentiable almost everywhere, the derivative $f'$ is Lebesgue integrable and
\[
f(t)-f(a)=\int_a^t f'(s) \ \textrm{d}s
\]
for every $t\in [a,b]$. 
The \emph{length} of a continuous curve $\gamma \colon [a,b] \to V$ is defined as
\[
l(\gamma):=\sup \sum^n_{i=1} ||\gamma(t_i)-\gamma(t_{i-1})||
\]
where the supremum ranges over all $n\in \mathbb{N}$ and all $a=t_0\leq t_1\leq \dots \leq t_n=b$. The curve $\gamma$ is called \emph{rectifiable} if $l(\gamma)$ is finite. For a rectifiable curve $\gamma$ we define its \emph{length function} $s_\gamma\colon [a,b]\to [0, l(\gamma)]$  by
\[
s_\gamma(t)=l( \gamma|_{[a,t]}).
\]
The length function gives rise to a unique curve $\bar{\gamma}\colon [0, l(\gamma)]\to V$ such that \begin{equation}
    \bar{\gamma} \circ s_\gamma =\gamma.
\end{equation} 
The curve $\bar{\gamma}$ is called the \emph{unit-speed parametrization} of $\gamma$ because one has for every $t\in [0,l(\gamma)]$ that
\[
l(\bar{\gamma}|_{[a,t]})=t-a.
\]

A curve $\gamma\colon [a,b] \to V$ is called \emph{absolutely continuous} if it is rectifiable and the length function $s_\gamma$ is absolutely continuous.  
Absolutely continuous curves in a Banach space $V$ do not need to be differentiable almost everywhere unless $V$ has the Radon-Nikod\'ym property. Nevertheless, if $V$ is dual to a separable Banach space then absolutely continuous curves in $V$ are weak* differentiable almost everywhere in the sense of the following lemma.
\begin{lem}[{\cite[Lemma 2.8]{HT:08}}]\label{lemma:2.8}
Let $Y$ be a separable Banach space. Then for every absolutely continuous curve $\gamma \colon[a,b]\to Y^*$ 
there is a weak* measurable function $\gamma'\colon [a,b] \to Y^*$ such that for almost every $t\in [a,b]$ and every $y\in Y$ one has
\begin{equation}\label{eq:w*-derivative}
\left\langle y, \frac{\gamma(t+h)-\gamma(t)}{h} \right\rangle \to \left\langle y, \gamma'(t)\right\rangle \quad 
\text{ as } h\to 0.
\end{equation}
\end{lem}
If $t\in [a,b]$ is such that \eqref{eq:w*-derivative} holds for every $y\in Y$ then $\gamma$ is called \emph{weak* differentiable} at $t$ and $\gamma'(t)$ is called the \emph{weak* derivative} of $\gamma$ at $t$. By the next two lemmas weak* derivatives have desirable analytical and metric properties.
\begin{lem}
\label{lem:chain-abs}
Let $Y$ be a separable Banach space and $\gamma \colon [a,b] \to Y^*$ be absolutely continuous. Then for every $\varphi \in C^\infty_0 ((a,b))$ one has
\begin{equation}
\label{eq:w*chain}
   \int^b_a \frac{\partial \varphi}{\partial t} (t) \cdot \gamma(t) \ \textnormal{\textrm{d}}t=- \int^b_a \varphi(t) \cdot \gamma'(t) \ \textnormal{\textrm{d}} t
\end{equation}
in the sense of Gelfand integrals.
\end{lem}
Lemma 2.11 in \cite{HT:08} claims that the equality \eqref{eq:w*chain} holds in the sense of Bochner integrals. In general however, as the subsequent example shows, the weak* derivative of an absolutely continuous curve in $Y^*$ does not need to be essentially separably valued and hence the Bochner integral $\int^b_a \varphi(t) \cdot \gamma'(t) \ \textnormal{\textrm{d}} t$ may not be defined.  
\begin{ex}
Consider the curve $\gamma \colon [0,1] \to L^\infty([0,1])$ given by $(\gamma(t))(s)=|t-s|$. Then $\gamma$ is an isometric embedding and hence in particular absolutely continuous. Further $\gamma$ is weak* differentiable at every $t\in [0,1]$ with weak* derivative \[\gamma'(t)=-\chi_{(0,t)}+\chi_{(t,1)}.\] Thus \[||\gamma'(s)-\gamma'(t)||_\infty=2\] for every $t\neq s$ and hence $\gamma'\colon [0,1]\to L^\infty([0,1])$ cannot be essentially separably valued.
\end{ex}
\begin{proof}[Proof of Lemma~\ref{lem:chain-abs}]
Let $\varphi \in C^\infty_0((a,b))$ and $y\in Y$. For $t\in [a,b]$ we will denote $\gamma_{y}(t):=\langle y , \gamma(t) \rangle$. Then $\gamma_{y} \colon [a,b] \to \mathbb{R}$ is absolutely continuous and, by the classical product rule,
\begin{equation}
\label{eq:chain-abs}
    \int^b_a \frac{\partial \varphi}{\partial t} (t) \cdot \gamma_{y}(t) \ \textnormal{\textrm{d}}t=- \int^b_a \varphi(t) \cdot \gamma_{y}'(t) \ \textnormal{\textrm{d}} t.
\end{equation}
 Furthermore by \eqref{eq:w*-derivative} for almost every $t\in [a,b]$ one has \begin{equation}
 \label{eq:deriv-eval}
\langle y, \gamma'(t) \rangle =\gamma_{y}'(t).
 \end{equation}
By \eqref{eq:chain-abs} and \eqref{eq:deriv-eval}, and because $y\in Y$ was arbitrary, we conclude equality \eqref{eq:w*chain}.
\end{proof}
\begin{lem}
\label{deriv:charac}
Let $Y$ be a separable Banach space. If $\gamma \colon[a,b]\to Y^*$ is absolutely continuous then
\begin{equation}\label{eq:norm-w*-derivative}
||\gamma '(t)||=\lim_{h\to 0} \frac{||\gamma(t+h)-\gamma(t)||}{|h|}=s_\gamma'(t)
\end{equation}
for almost every $t\in [a,b]$.
\end{lem}

\begin{proof}
Assume $t\in [a,b]$ is such that $s_\gamma$ is differentiable at $t$ and that $\gamma$ is weak* differentiable at $t$. Then for every $y\in Y$ with $||y||\leq 1$ one has
\[
\langle y , \gamma'(t)\rangle = \lim_{h\to 0} \left\langle y, \frac{\gamma(t+h)-\gamma(t)}{h} \right\rangle \leq \liminf_{h \to 0} \frac{||\gamma(t+h)-\gamma(t)||}{|h|}
\]
and hence
\begin{equation}
\label{eq:md1}
||\gamma'(t)||\leq \liminf_{h \to 0} \frac{||\gamma(t+h)-\gamma(t)||}{|h|}.
\end{equation}
Furthermore
\begin{equation}
\label{eq:md2}
\limsup_{h\to 0} \frac{||\gamma(t+h)-\gamma(t)||}{|h|}\leq \limsup_{h\to 0} \frac{l(\gamma|_{[t,t+h]})}{|h|}=s_\gamma'(t).
\end{equation}
To prove the reverse inequalities, let $t,\bar{t}\in [a,b]$ with $t<\bar{t}$. Then for every $y\in Y$ with $||y||\leq 1$ one has
\begin{equation}
    \langle y , \gamma(\bar{t})-\gamma(t) \rangle =\int_t^{\bar{t}} \langle y, \gamma'(r) \rangle \ \textrm{d}r\leq \int_t^{\bar{t}} ||\gamma'(r)|| \ \textrm{d} r
\end{equation}
and thus
\begin{equation}
\label{eq:md3}
    ||\gamma(\bar{t})-\gamma(t)||\leq \int_t^{\bar{t}} ||\gamma'(r)|| \ \textrm{d} r.
\end{equation}
Since $t$ and $\bar{t}$ were arbitrary, we conclude that
\begin{equation}
\label{eq:md4}
\int_a^{b} s_\gamma'(r) \ \textrm{d} r=l(\gamma)\leq \int_a^{b} ||\gamma'(r)|| \ \textrm{d} r.
\end{equation}
\eqref{eq:md1}, \eqref{eq:md2} and \eqref{eq:md4} together imply the claim.
\end{proof}

\section{Banach space valued Sobolev maps}
\label{sec:3}
Throughout this section let $\Omega \subset \mathbb{R}^n$ be open, $V$ be a Banach space, $Y$ be a separable Banach space and $p\in [1,\infty)$.

\subsection{The Reshetnyak-Sobolev space}
\label{sec:3.1}

The following definition of first-order So\-bo\-lev functions with values in Banach spaces goes back to \cite{Reshetnyak97}.
\begin{defn}
\label{definition:RS-space}
The \textit{Reshetnyak--Sobolev space} $R^{1,p}(\Omega; V)$ consists of those functions $f\in L^p(\Omega;V)$ such that:
\begin{enumerate}[(A)]
\item for every $v^*\in V^*$ the function $x \mapsto \langle v^*, f(x) \rangle$ lies in the classical Sobolev space $W^{1,p}(\Omega):=W^{1,p}(\Omega ; \mathbb{R})$;
\item \label{item:RS2} there is a function $g\in L^p(\Omega)$ such that for every $v^*\in V^*$ one has
\[ |\nabla\langle v^*, f(x) \rangle| \leq ||v^*||\cdot g(x) \quad \textnormal{for a.e.\ }x\in \Omega.\]
\end{enumerate}
A function $g$ as in (\ref{item:RS2}) will be called a \emph{weak upper gradient} of $f$. A seminorm is defined on $R^{1,p}(\Omega ;V)$ by 
\[
||f||_{R^{1,p}}:=\left(\int_\Omega ||f(x)||^p \ \textrm{d}x\right)^{1/p} + \inf_{g} \ ||g||_{L^p}
\]
where $g$ ranges over all weak upper gradients of $f$.
\end{defn}
Indeed, Definition~\ref{definition:RS-space} is a variation on the original definition by Reshetnyak. The reason for the present choice of definition is that, in contrast to the definition in \cite{Reshetnyak97}, it also allows for unbounded domains $\Omega$. This extension is possible because we limit ourselves here to maps with values in Banach spaces while Reshetnyak considers general metric target spaces. In any case the two definitions are equivalent if $\Omega$ is a bounded domain, see \cite[Lemma 2.16]{HT:08} and \cite[Theorem 5.1]{Reshetnyak97}.

To prove that $R^{1,p}(\Omega ;Y^*)$ equals $W^{1,p}_*(\Omega ; Y^*)$, we will work with the following auxiliary definition that interpolates between the two spaces.
\begin{defn}
The space $R^{1,p}_*(\Omega; V^*)$ consists of those functions $f\in L^p(\Omega;V^*)$ such that:
\begin{enumerate}[(A*)]
\item for every $v\in V$ the function $x \mapsto \langle v, f(x) \rangle$ lies in $W^{1,p}(\Omega)$;
\item \label{item:RS2*} there is a function $g\in L^p(\Omega)$ such that for every $v\in V$ one has
\[ |\nabla\langle v, f(x) \rangle| \leq ||v||\cdot g(x) \quad \textnormal{for a.e.\ }x\in \Omega.\]
\end{enumerate}
A function $g$ as in (B*) 
will be called a \emph{weak* upper gradient} of $f$. A seminorm is defined on $R^{1,p}_*(\Omega ;V^*)$ by 
\[
||f||_{R^{1,p}_*}:=\left(\int_\Omega ||f(x)||^p \ \textrm{d}x\right)^{1/p} + \inf_{g} \ ||g||_{L^p}
\]
where $g$ ranges over all weak* upper gradients of $f$. 
\end{defn}

We will denote by $\textnormal{ACL}(\Omega)$ the collection of all functions $f\colon \Omega \to \mathbb{R}$ for which the restriction of $f$ to almost every compact line segment, that is contained in $\Omega$ and parallel to some coordinate axis, is absolutely continuous. Recall that every real valued Sobolev function in $f\in W^{1,p}(\Omega)$ has a representative $\widetilde{f}\in \textnormal{ACL}(\Omega)$. The following lemma shows that similar is true for functions in $R^{1,p}_*(\Omega; Y^*)$.
\begin{lem}\label{lem:lemma2.13}
Let $V$ be a Banach space and $f\in R^{1,p}_{*}(\Omega; V^*)$.
Then for every $j\in \{1, \dots ,n\}$ the function $f$ has a representative $\widetilde{f}^j$ that is absolutely continuous on almost every compact line segment which is contained in $\Omega$ and parallel to the $x_j$-axis. 
Moreover, for every weak* upper gradient $g$ of $f$ one has
\begin{equation}\label{eq:lemma2.13est}
\lim_{h\to 0}\frac{\|\tilde f^j(x+h e_j)- \tilde f^j(x)\|}{|h|} \leq g(x) \quad \textnormal{for a.e.\ } x\in \Omega.
\end{equation} 
\end{lem}
Lemma~\ref{lem:lemma2.13} generalizes Lemma 2.13 in \cite{HT:08} from $R^{1,p}$ to $R^{1,p}_*$. A posteriori Proposition~\ref{prop:R*=R} will show that this is not a proper generalization.
\begin{proof}
Fix $j\in \{1, \dots , n\}$ and a weak* upper gradient $g$ of $f$. Since $f\in L^p(\Omega; V^*)$, there is a nullset $\Sigma_0\subset \Omega$ such that $f(\Omega\setminus\Sigma_0)$ is separable.
Let $(v^*_i)_{i\in\mathbb N}$ be a dense sequence in the difference set
$f(\Omega\setminus\Sigma_0) - f(\Omega\setminus\Sigma_0)$.
For each $i\in \mathbb{N}$ let 
$(v_{ik})_{k\in\mathbb N}$ be a sequence of unit vectors in $V$ such that
$\|v^*_i\| = \lim_{k\to\infty}\langle v_{ik}, v^*_i \rangle$.
Then for every $i,k\in \mathbb{N}$ one has $\langle v_{ik}, f \rangle\in W^{1,p}(\Omega)$ and 
\begin{equation}
\label{eq:ik-upper}
|\nabla \langle v_{ik}, f(x)\rangle|\leq g(x) \quad \textnormal{for a.e.\ } x\in \Omega.
\end{equation}
Denote by $f_{ik}$ a representative of $\langle v_{ik}, f \rangle$ that is in $\textnormal{ACL}(\Omega)$ and by $\Sigma_{ik}$ the null set on which $f_{ik}$ differs from 
$\langle v_{ik}, f \rangle$.  
Then for almost every line segment
$l:[a,b]\to\Omega$ that is parallel to the $x_j$-axis one has:
\begin{enumerate}[(i)]
\item  \label{item:i} $g$ is integrable over $l$;
\item  \label{item:ii} $\mathcal{H}^1(l\cap\Sigma) = 0$ where 
$\Sigma = \Sigma_0\cup\bigcup_{i,k}\Sigma_{ik}$;
\item \label{item:iii} for every $i,k \in\mathbb N$ and every $a\leq s\leq t\leq b$
\begin{equation}\label{eq:lemma2.13-condition-c}
|f_{ik}(l(t)) - f_{ik}(l(s))| \leq \int_s^tg(l(\tau)) \ \textrm{d}\tau.
\end{equation}
\end{enumerate}    
The Fubini theorem ensures (\ref{item:i}) and (\ref{item:ii}), while (\ref{item:iii}) follows by \eqref{eq:ik-upper}.

Let $l\colon [a,b] \to \Omega$ be a line segment parallel to the $x_j$-axis for which the properties (\ref{item:i}), (\ref{item:ii}) and (\ref{item:iii}) are satisfied. For given $s,t\in l^{-1}(\Omega \setminus \Sigma)$ with $s\leq t$
there is a subsequence $(v^*_{i_m})$ that converges to $f(l(t)) - f(l(s))$ in $V^*$.
Thus, we have
\begin{align}
\label{eq:R*abs}
&\|f(l(t)) - f(l(s))\|\\
= &\lim_{m\to\infty}\|v^*_{i_m}\|
=\lim_{m\to\infty}\lim_{k\to\infty}\langle v_{i_mk}, v^*_{i_m} \rangle\\
= &\limsup_{m\to\infty}\limsup_{k\to\infty}
\big(\langle v_{i_mk}, v^*_{i_m}-(f(l(t)) - f(l(s)))\rangle\\
& \ \ \ \ \ \ \ \ \ \ \ \ \ \ \ \ \ \  \ \ \ \ \ \ \ \ \ \ \ \ + \langle v_{i_mk}, f(l(t)) - f(l(s))\rangle \big)\\
\leq &\limsup_{m\to\infty}\limsup_{k\to\infty}
\big(\|v^*_{i_m}-(f(l(t)) - f(l(s)))\|\\
& \ \ \ \ \ \  \ \ \ \ \ \ \ \ \ \ \ \ \ \ \ \ \ \ \ \ \ \ \ \ + |f_{i_mk}(l(t)) - f_{i_mk}(l(s))| \big) \\
\leq &\int_s^tg(l(\tau)) \ \textrm{d}\tau.
\end{align}
In particular, by properties (\ref{item:i}) and (\ref{item:ii}), and inequality \eqref{eq:R*abs} the restriction of $f$ to $l$ has a unique $\mathcal{H}^1$-representative that is absolutely continuous. The uniqueness implies that these representatives coincide where different line segments overlap. Hence we conclude that $f$ has a representative $\widetilde{f}^j$ that is absolutely continuous on every compact line segment $l$ that satisfies the properties (\ref{item:i}), (\ref{item:ii}) and (\ref{item:iii}). 
Furthermore, by \eqref{eq:R*abs} for every such $l$ one has 
\[
\|\widetilde{f}^j(l(t)) - \widetilde{f}^j(l(s))\| \leq \int_s^tg(l(\tau)) \ \textrm{d}\tau
\]
and hence we conclude that \eqref{eq:lemma2.13est} is satisfied.
\end{proof}
Given that in general $W^{1,p}_*(\Omega ;V^*)$ does not equal $W^{1,p}(\Omega ;V^*)$ the following proposition might be a bit surprising.
\begin{prop}
\label{prop:R*=R}
Let $V$ be a Banach space. Then
\[
R^{1,p}_*(\Omega ; V^*)=R^{1,p}(\Omega ; V^*) 
\]
with $||\cdot ||_{R^{1,p}_*}\leq ||\cdot ||_{R^{1,p}}\leq \sqrt{n} ||\cdot ||_{R^{1,p}_*}$.
\end{prop}
\begin{proof}
Trivially $R^{1,p}(\Omega; V^*)\subseteq R^{1,p}_*(\Omega ;V^*)$, and $||f||_{R^{1,p}_*}\leq ||f||_{R^{1,p}}$ for functions $f\in R^{1,p}(\Omega ; V^*)$. 
For the other inclusion let $f\in R^{1,p}_*(\Omega ; V^*)$ and $g$ be a weak* upper gradient of $f$. Since $f\in L^{p}(\Omega; V^*)$, for $v^{**}\in V^{**}$ the function $f_{v^{**}}=\langle v^{**}, f\rangle$ lies in $L^{p}(\Omega)$. For $j\in \{1, \dots, n\}$ let $\widetilde{f}^j$ be a representative of $f$ as in Lemma~\ref{lem:lemma2.13}. Then $\widetilde{f}^j_{v^{**}}:=\langle v^{**},\widetilde{f}^j \rangle $ is a representative of $f_{v^{**}}$ that is absolutely continuous on almost every compact line segment parallel to the $x_j$-axis. Thus $\widetilde{f}^j_{v^{**}}$ is almost everywhere partial differentiable in the $x_j$-direction. By the product rule and the Fubini theorem it follows that
\begin{equation}
   \int_\Omega \frac{\partial \varphi}{\partial x_j}(x) \cdot f_{v^{**}}(x)\ \textrm{d} x= \int_\Omega \frac{\partial \varphi}{\partial x_j}(x) \cdot \widetilde{f}^j_{v^{**}}(x) \ \textrm{d}x= \int_\Omega \varphi(x) \cdot \frac{\partial \widetilde{f}^j_{v^{**}}}{\partial x_j}(x) \ \textrm{d} x 
\end{equation}
for every $\varphi \in C^\infty_0(\Omega)$. In particular $\frac{\partial \widetilde{f}^j_{v^{**}}}{\partial x_j}$ is a $j$-th weak partial derivative of~$f_{v^{**}}$. Furthermore, by Lemma~\ref{lem:lemma2.13} at almost every $x\in \Omega$ one has
\begin{equation}
\left| \frac{\partial \widetilde{f}^j_{v^{**}}}{\partial x_j}(x)\right|\leq \lim_{h \to 0} \frac{\|\tilde f^j(x+h e_j)- \tilde f^j(x)\|}{|h|}\leq g(x)
\end{equation}
and hence
\begin{equation}
    |\nabla f_{v^{**}} (x)|  = \left( \sum_{j=1}^n \left(\frac{\partial \widetilde{f}^j_{v^{**}}}{\partial x_j}(x)\right)^2 \right)^{1/2}\leq \sqrt{n} \cdot g(x).
\end{equation}
Since $v^{**}\in V^{**}$ and the weak* upper gradient $g\in L^{p}(\Omega)$ were arbitrary, we conclude that 
$f\in R^{1,p}(\Omega ;V^{*})$  and \[\|f\|_{R^{1,p}}\leq \sqrt{n}\cdot \|f\|_{R^{1,p}_*}.\]
This completes the proof.
\end{proof}
\subsection{The Sobolev space $W^{1,p}_*$}
\label{sec:3.2}
Let $f\in W^{1,p}_*(\Omega ; V^*)$. We will denote by $\partial_j f$ the function $f_j$ as in Definition~\ref{fd} and call the vector $\nabla f(x) = (\partial_1f(x),\dots,\partial_nf(x))$ the \emph{weak* gradient} of $f$ at $x\in \Omega$. Further we define 
\[
|\nabla f(x)|:=\left(\sum_{i=0}^n ||\partial_i f(x)||^2\right)^{1/2}
\]
and a seminorm on $W^{1,p}_*(\Omega ; V)$ by
\[
||f||_{W^{1,p}_*}:= \left( \int_\Omega ||f(x)||^p \ \textrm{d}x \right)^{1/p}+ \left(\int_\Omega |\nabla f(x)|^p\ \textrm{d}x  \right)^{1/p}
\]
\begin{prop}
\label{prop:W*=R*}
Let $Y$ be a separable Banach space. Then
\[
W^{1,p}_*(\Omega ; Y^*)=R^{1,p}_*(\Omega ; Y^*) 
\]
with $||\cdot ||_{R^{1,p}_*}\leq ||\cdot ||_{W^{1,p}_*}\leq \sqrt{n} ||\cdot ||_{R^{1,p}_*}$.
\end{prop}
\begin{proof}
Let $f\in W^{1,p}_*(\Omega ; Y^*)$. Since $f\in L^p(\Omega ; Y^*)$ we know that for $y\in Y$ the function $f_y:=\langle y, f\rangle$ lies in $L^p(\Omega)$. Further, by definition of the Gelfand integral, for $j\in \{1, \dots ,n \}$ the function $\langle y, \partial_j f\rangle$ is a $j$-th weak partial derivative of $f_j$. Hence $f_y \in R^{1,p}_*(\Omega ; Y^*)$ and 
\[
|\nabla f_y (x)|= \left( \sum_{j=1}^n  \langle y, \partial_j f(x)\rangle^2\right)^{1/2}\leq \left( \sum_{j=1}^n ||\partial_j f(x) ||^2\right)^{1/2}=|\nabla f(x)|
\]
for a.e.\ $x\in \Omega$. In particular $f\in R^{1,p}_*(\Omega ; Y^*)$ and $|\nabla f|$ is a weak* upper gradient of $f$. The latter also implies  $||f||_{R^{1,p}_*}\leq ||f||_{W^{1,p}_*}$.

Now, for the other inclusion, let $f\in R^{1,p}_*(\Omega ; Y^*)$ and $g$ be a weak* upper gradient of $f$. For $j\in \{1, \dots ,n\}$ let $\widetilde{f}^j$ be a representative of $f$ as in Lemma~\ref{lem:lemma2.13}. 
Define $f_j(x)$ as the weak* partial derivative~$\frac{\partial \widetilde{f}^j}{\partial x_j}(x)$,
which is defined almost everywhere due to Lemma~\ref{lemma:2.8}. 
Then the function $f_j\colon \Omega \to Y^*$ is weak* measurable. Furthermore, by Lemma~\ref{lem:chain-abs} and the Fubini theorem, for every $\varphi \in C^\infty_0(\Omega)$ one has
\begin{equation}
\label{eq:main-partial}
\int_\Omega \frac{\partial \varphi}{\partial x_j}(x) \cdot f(x) \ \textrm{d}x=\int_\Omega \frac{\partial \varphi}{\partial x_j}(x) \cdot \widetilde{f}^j(x) \ \textrm{d}x=\int_\Omega \varphi(x) \cdot f_j(x) \ \textrm{d} x
\end{equation}
in the sense of Gelfand integrals. Also, by Lemmas~\ref{deriv:charac} and \ref{lem:lemma2.13}, 
\[
\|f_j(x)\|\leq g(x)\quad \textnormal{for a.e.\ }x\in \Omega.
\] 
In particular, since $g\in L^p(\Omega)$, we conclude that $f_j\in L^p_*(\Omega ; Y^*)$ and hence by \eqref{eq:main-partial} that $f\in W^{1,p}(\Omega ; Y^*)$ with
\[
|\nabla f (x) | =\left( \sum_{j=1}^n \|f_j(x)\|^2\right)^{1/2} \leq \sqrt{n}\cdot g(x)
\]
for almost every $x\in \Omega$. Since $g$ was an arbitrary weak* upper gradient of $f$ it also follows that $||f||_{W^{1,p}_*}\leq \sqrt{n} ||f||_{R^{1,p}_*}$.
\end{proof}
Propositions~\ref{prop:R*=R} and~\ref{prop:W*=R*} together imply the following quantitative version of Theorem~\ref{theorem:W=R}.
\begin{thm}
\label{thm:W=Rquant}
Let $Y$ be a separable Banach space. Then
\[
W^{1,p}_*(\Omega ; Y^*)=R^{1,p}(\Omega ; Y^*) 
\]
with $\tfrac{1}{\sqrt{n}}||\cdot ||_{R^{1,p}}\leq ||\cdot ||_{W^{1,p}_*}\leq \sqrt{n} ||\cdot ||_{R^{1,p}}$.
\end{thm}
It has been shown in \cite{CJPA2020} (see also \cite{Evs:2020}) that $W^{1,p}(\Omega ; V)=R^{1,p}(\Omega ; V)$ if and only if $V$ has the Radon-Nikodým property. Concerning Theorem~\ref{thm:W=Rquant}, it seems to be a natural conjecture that conversely the equality $W^{1,p}_*(\Omega ; V^*)=R^{1,p}(\Omega ; V^*)$ implies that $V$ is separable.

\section{Metric space valued Sobolev maps}
\label{sec:4}
Throughout this section let $\Omega\subset \mathbb{R}^n$ be a bounded domain, $X=(X,d)$ be a complete metric space and $p\in [1,\infty)$.
\subsection{The Sobolev space $W^{1,p}_*(\Omega ;X)$} \label{sec:4.1}
The \emph{Reshetnyak-Sobolev space} can be defined as
\begin{equation}
\label{def:R-met}
R^{1,p}(\Omega ; X):=\{f\colon \Omega \to X \ | \ \iota \circ f \in R^{1,p}(\Omega ; V)\}
\end{equation}
where $\iota \colon X \to V$ is any fixed isometric embedding of $X$ into a Banach space $V$. By the following example such embedding $\iota$ always exists.
\begin{ex}
Let $X$ be a metric space. Denote by $\ell^\infty(X)$ the Banach space of bounded functions $f\colon X\to \mathbb{R}$ with norm given by
\[
||f||_\infty:= \sup_{z\in X} |f(z)|.
\]
Then for given $z_0\in X$ the function $\bar{\kappa} \colon X \to \ell^\infty(X)$ given by 
\[
(\bar{\kappa}(z))(w):=d(z,w)-d(w,z_0)
\]
defines an isometric embedding, see e.g.\ \cite[p.\ 5]{Hei:03}.
\end{ex}
Furthermore, under the present assumption that $\Omega$ is bounded, the definition of $R^{1,p}(\Omega ;X)$ does not depend on the chosen embedding~$\iota$ and is equivalent to the original definition by Reshetnyak, see \cite[Lemma 2.16]{HT:08} and \cite[Theorem 5.1]{Reshetnyak97}. Thus Theorem~\ref{theorem:W=R} has the following consequence.
\begin{thm}
\label{thm:W=Rmet}
Let $\Omega \subset \mathbb{R}^n$ be a bounded domain, $X$ be a complete metric space, $Y$ be a separable Banach space and $\iota \colon X \to Y^*$ be an isometric embedding. Then
\begin{equation}
    R^{1,p}(\Omega ;X)=\left\{ f \colon \Omega \to X \ |\ \iota\circ f \in W^{1,p}_*(\Omega ; Y^*) \right\}.
\end{equation}
\end{thm}
Certainly not every metric space $X$ isometrically embeds into the dual of a separable Banach space. A simple obstruction is the cardinality of $X$ which must be bounded above by $2^{2^\omega}$.
For a separable metric space $X$ however, due to the following example, there is always an isometric embedding as in Theorem~\ref{thm:W=Rmet}.
\begin{ex}
Let $X$ be a separable metric space and $(z_i)_{i\in \mathbb{N}}$ be a dense sequence of points in $X$. Denote $\ell^\infty:=\ell^\infty(\mathbb{N})$. Then $\ell^\infty$ is the dual of the separable Banach space $\ell^1:=\ell^1(\mathbb{N})$. The function $\kappa \colon X \to \ell^\infty$ given by \[\kappa(z):=\left(d(z,z_i)-d(z_i,z_1)\right)_{i\in \mathbb{N}}\]
is called the \emph{Kuratowski embedding} of $X$. It is not hard to check that $\kappa$ defines an isometric embedding, see e.g.\ \cite[p.\ 11]{Hei:03}. 
\end{ex}
Thus, for a bounded domain $\Omega$ and a complete separable metric space $X$, one can define
\begin{equation}
\label{eq:def-W*-met}
W^{1,p}_*(\Omega ; X):=\left\{ f \colon \Omega \to X \ |\ \kappa \circ f \in W^{1,p}_*(\Omega ; \ell^\infty) \right\}
\end{equation}
and deduce from Theorem~\ref{thm:W=Rmet} that $W^{1,p}_*(\Omega ; X)=R^{1,p}(\Omega ; X)$. The assumption that $\Omega$ is bounded is needed to ensure that $W^{1,p}_*(\Omega ; X)$ is well-defined by means of \eqref{eq:def-W*-met} and does not depend on the concrete choice of Kuratowski embedding. For a non-separable complete metric space $X$ one can define $W^{1,p}_*(\Omega ;X)$ as the union of the spaces $W^{1,p}_*(\Omega ;S)$ where $S$ ranges over all separable closed subsets of $X$. Since Sobolev functions are essentially separably valued, also for such non-separable $X$, Theorem~\ref{thm:W=Rmet} implies that $W^{1,p}_*(\Omega ;X)\subset R^{1,p}(\Omega ;X)$ and that every function in $R^{1,p}(\Omega ;X)$ has a representative in $W^{1,p}_*(\Omega ;X)$.
\subsection{The Sobolev space $W^{1,p}(\Omega, X)$}
\label{sec:4.2}
The aim of this subsection is to prove Theorem~\ref{thm:triv}. To this end let $X$ be a complete separable metric space, $(x_i)_{i\in \mathbb{N}}$ be a dense sequence of points in $X$ and $\kappa \colon X\to \ell^\infty$ be the corresponding Kuratowski embedding. The key step for the proof is the following lemma.
\begin{lem}
\label{lem:nonexcurve}
 If $\gamma \colon [a,b]\to \kappa(X)\subset \ell^\infty$ is a non-constant absolutely continuous curve then the weak* derivative $ \gamma'\colon [a,b]\to \ell^\infty$ is not essentially separably valued.
\end{lem}
\begin{proof}
Since $\gamma$ is non-constant we have $l:=l(\gamma)>0$. As in Section~\ref{sec:abs-curves} we factorize $\gamma=\bar{\gamma}\circ s_\gamma$ where $\bar{\gamma} \colon [0,l]\to \kappa (X)$ is the unit-speed parametrization of $\gamma$ and $s_\gamma \colon [a,b] \to [0,l]$ is the length function of $\gamma$. First we show that $\bar{\gamma}'\colon [0,l] \to \ell^\infty$ is not essentially separably valued.

By Lemma~\ref{deriv:charac} for a.e.\ $t\in [0,l]$ one has that
\begin{equation}
\label{eq:md=1}
   ||\bar{\gamma}'(t)||_\infty=\lim_{h\to 0} \frac{||\bar{\gamma}(t+h)-\bar{\gamma}(t)||_\infty}{|h|}=1.
\end{equation}
Let $E$ be the set of points $t_0\in (0,l)$ at which $\bar{\gamma}$ is weak* differentiable and \eqref{eq:md=1} holds. By Theorem~\ref{thm:approx}, to show that $\bar{\gamma}'$ is not essentially separably valued, it suffices to prove that $\bar{\gamma}'$ is not approximately continuous at every $t_0\in E$.

So fix $t_0 \in E$ and let $h_0 >0$ be so small that for any $h\in \mathbb{R}$ with $|h|\leq h_0$ one has
\begin{equation}
\label{eq:choice1}
\frac{1}{2} \cdot |h| < ||\gamma(t_0+h)-\gamma(t_0)||_\infty .
\end{equation}
Further fix some arbitrary $0<h<h_0$ and accordingly choose $i\in \mathbb{N}$ such that
\begin{equation}
\label{eq:choice2}
||\kappa (x_i)-\gamma(t_0)||_\infty \leq \frac{1}{4} \cdot h.
\end{equation}
By Lemma~\ref{lemma:2.8} for every point $t\in [0,l]$ at which $\bar{\gamma}$ is weak* differentiable one has
\begin{equation}
\label{eq:coord-rep}
\bar{\gamma}'(t)=(\bar{\gamma}'_i(t))_{i\in \mathbb{N}}
\quad \textnormal{where } \bar{\gamma}(t)=(\bar{\gamma}_i(t))_{i\in \mathbb{N}}
\end{equation}is the coordinate representation of $\bar{\gamma}$. From the fundamental theorem of calculus, the definition of the Kuratowski embedding, \eqref{eq:choice1} and \eqref{eq:choice2} it follows that
\begin{align}
\label{eq:lb}
    \int^{t_0 +h}_{t_0} \bar{\gamma}_i'(t) \ \textrm{d}t&=\bar{\gamma}_i(t_0+h)-\bar{\gamma}_i(t_0)\\
    &=||\bar{\gamma}(t_0+h)-\kappa (x_i)||_\infty-||\bar{\gamma}(t_0)-\kappa(x_i)||_\infty \\ &\geq \frac{1}{4}\cdot h
\end{align}
Since $|\bar{\gamma}_i'(t)|\leq 1$ for a.e.\ $t$, this implies that
\begin{equation}
\label{eq:Fh+}
    \mathcal{L}^1\left( F^+_h\right)\geq \frac{1}{8} \cdot h \quad 
\textnormal{where }
    F^+_h:=\left\{ t\in (t_0,t_0+h) \ :\  \bar{\gamma}_i'(t)\geq \frac{1}{8}\right\}.
\end{equation}
Similarly
\begin{equation}
\int_{t_0-h}^{t_0} \bar{\gamma}_i'(t) \ \textrm{d} t\leq -\frac{1}{4} \cdot h
\end{equation}
and hence
\begin{equation}
\label{eq:Fh-}
    \mathcal{L}^1\left(F^-_h\right)\geq \frac{1}{8} \cdot h \quad 
\textnormal{where }
 F^-_h:=\left\{ t\in (t_0-h,t_0) \ :\  \bar{\gamma}_i'(t)\leq -\frac{1}{8}\right\}.
\end{equation}
Note that for every $t^+\in F_h^+\cap E$ and $t^- \in F_h^-\cap E$ one has
\begin{equation}
\label{eq:apart}
||\bar{\gamma}'(t^+)-\bar{\gamma}'(t^-)||_\infty \geq |\bar{\gamma}_i'(t^+)-\bar{\gamma}_i'(t^-)|\geq \frac{1}{4}.
\end{equation}
Since $0<h<h_0$ was arbitrary, \eqref{eq:Fh+}, \eqref{eq:Fh-} and \eqref{eq:apart} together imply that $\bar{\gamma}'$ cannot be approximately continuous at $t_0$. In turn, because $t_0\in E$ was arbitrary, we conclude from Theorem~\ref{thm:approx} that $\bar{\gamma}'$ is not essentially separably valued.

Now let $N\subset [a,b]$  be an arbitrary nullset. We need to show that $\gamma'([a,b]\setminus N)$ is not separable. By Lemma~\ref{lemma:2.8}, after possibly passing to a larger null set, we may assume that for every $t\in [a,b]\setminus N$ the curve $\gamma$ is weak* differentiable at $t$ and the function $s_\gamma$ is differentiable at $t$. Note that \begin{equation}
\label{eq:areaformula}
\mathcal{L}^1(s_\gamma(A))=\int_A s_\gamma'(t)\ \textrm{d} t
\end{equation}
for every measurable subset $A\subset [a,b]$. Thus, we may further assume that for every $t\in [a,b]\setminus N$ either $\bar{\gamma}$ is weak* differentiable at $s_\gamma(t)$ or $s_\gamma'(t)=0$. In particular, it follows that \begin{equation}
\label{eq:chain-n-p}
(\bar{\gamma}'\circ s_\gamma)(t) \cdot s_\gamma'(t)=\gamma'(t)
\end{equation}
on $[a,b]\setminus N$. 
By \eqref{eq:areaformula} one has that $M:=s_\gamma(N)\cup s_\gamma(\{s_\gamma'=0\})$ is a null set and hence $\bar{\gamma}'([0,l]\setminus M)$ is not separable. On the other hand, $s_\gamma$ is surjective and hence by \eqref{eq:chain-n-p} it follows that \[\bar{\gamma}'([0,l]\setminus M)\subset \langle \gamma'([a,b]\setminus N) \rangle_{\mathbb{R}} \]
where $\langle \gamma'([a,b]\setminus N) \rangle_{\mathbb{R}}$ denotes the linear span of $\gamma'([a,b]\setminus N)$ in $\ell^\infty$.
In particular, the linear span of $\gamma'([a,b]\setminus N)$ is not separable and hence also $\gamma'([a,b]\setminus N)$ itself cannot be separable
\end{proof}
Recall that we have defined
\begin{equation}
\label{eq:def-med}
W^{1,p}(\Omega ;X):=\{f \colon \Omega \to X \ : \ \kappa \circ f \in W^{1,p}(\Omega ;\ell^\infty)\}.
\end{equation}
\begin{proof}[Proof of Theorem~\ref{thm:triv}]
Let $f\colon \Omega \to X$ be almost everywhere constant. Then also $\kappa \circ f$ is almost everywhere constant. Hence $\kappa \circ f$ is measurable and essentially separably valued. Since $\Omega$ is bounded,  one has that $x\mapsto ||f(x)||_\infty$ lies in $L^p(\Omega)$. Thus, choosing $f_j$ identically zero for each $j$, we conclude that $\kappa \circ f \in W^{1,p}(\Omega ; \ell^\infty)$ and hence that $f\in W^{1,p}(\Omega ;X)$.

For the other inclusion let $f\in W^{1,p}(\Omega ;X)$. Then, by definition, $h:=\kappa \circ f$ lies in $W^{1,p}(\Omega ; \ell^\infty)$. Trivially this implies that $h \in W^{1,p}_*(\Omega ;\ell^\infty)$ and that $\partial_j h$ lies in $L^p(\Omega ;X)\subset L^p_*(\Omega ;X)$ for each $j$.  Since $W^{1,p}_*(\Omega ; \ell^\infty)$ equals $R^{1,p}_*(\Omega ; \ell^\infty)$, Lemma~\ref{lem:lemma2.13} implies that for each $j$ the function $h$ has a representative $\widetilde{h}^j$ that is absolutely continuous on almost every compact line segment parallel to the $x_j$-axis. In particular, there is a nullset $N\subset \Omega$ such that $\partial_j h(\Omega \setminus N)$ is separable for every $j$. Note that, since $X$ is complete, for almost every compact line segment $l\colon [a,b] \to \Omega$ parallel to the $x_j$-axis the image $\widetilde{h}^j \circ l([a,b])$ must be contained in $\kappa(X)$. Further the proof of Proposition~\ref{prop:W*=R*} shows that, possibly enlarging $N$, we can assume that for each $j$ one has \[\partial_j h (x)=\frac{\partial \widetilde{h}^j}{\partial x_j}(x)\] for every $x\in\Omega\setminus N$.

Assume $f$ was not almost everywhere constant. Since $\Omega$ is connected, this implies that there is some $j$ such that not for almost every line segment parallel to the $x_j$-axis the restriction of $f$ to the line segment is constant. Hence we can find a line segment $l\colon [a,b]\to \Omega$ such that
\begin{enumerate}[(i)]
    \item $\mathcal{H}^1(l([a,b]\cap N))=0$,
    \item  $\widetilde{h}^j\circ l([a,b])\subset \kappa(X)$, and
    \item  $\widetilde{h}^j\circ l$ is absolutely continuous and non-constant.
\end{enumerate}
By Lemma~\ref{lem:nonexcurve}, $(\widetilde{h}^j\circ l)'\colon [a,b] \to X$ cannot be essentially separably valued. This gives a contradiction because 
\[
(\widetilde{h}^j\circ l)'(t)=\partial_j h(l(t))
\]
for every $t\in l([a,b])\setminus N$ and $\partial_j h(\Omega \setminus N)$ is separable. 
\end{proof}
Note that, by Theorem~\ref{thm:triv}, if $X$ is a separable Banach space then the definition of $W^{1,p}(\Omega ;X)$ given in \eqref{eq:def-med} is not compatible with the one given in Definition~\ref{fd}. For example, most trivially, one may consider the case $X=\mathbb{R}$ where Definition~\ref{fd} gives the classical Sobolev space $W^{1,p}(\Omega)$.

\bibliographystyle{plain}
\bibliography{bibliography}
\end{document}